\documentclass[11pt,letterpaper,reqno]{amsart}
\usepackage{amssymb}
\usepackage{amsmath}
\usepackage{amsthm}
\usepackage{amsfonts}
\usepackage{bbm}
\usepackage{enumitem} 
\usepackage{booktabs}
\usepackage{graphicx}
\usepackage{xcolor}
\usepackage[T1]{fontenc}
\usepackage{doi}
\usepackage{comment} 
\usepackage{hyperref}
\hypersetup{pdfstartview={FitH}}
\usepackage{bookmark}
\usepackage[capitalize,noabbrev]{cleveref}

\newtheorem{theorem}{Theorem}[section]
\newtheorem{lemma}[theorem]{Lemma}

\newtheorem{corollary}[theorem]{Corollary}

\theoremstyle{definition}

\newtheorem{definition}[theorem]{Definition}

\newcommand\doublecheck{\textcolor{blue}{\checked\kern-0.6em\checked}}
\newcommand{\newvtheorem}[2]{\newtheorem{#1}[theorem]{\llap{\textnormal{\doublecheck} }#2}}
\newcommand{\newvtheoremstar}[2]{\newtheorem*{#1}{\llap{\textnormal{\doublecheck} }#2}}
\newvtheorem{vtheorem}{Theorem}
\newvtheorem{vlemma}{Lemma}
\newvtheoremstar{vtheoremstar}{Theorem}
\newvtheoremstar{vlemmastar}{Lemma}

\begin{document}

\title{Gaps in multiplicative Sidon sets II}

\author{Wouter van Doorn}
\address{Groningen, the Netherlands}
\email{wonterman1@hotmail.com}

\author{Quanyu Tang}
\address{School of Mathematics and Statistics, Xi'an Jiaotong University, Xi'an 710049, P. R. China}
\email{tangquanyu827@gmail.com}



\begin{abstract}
With $\rho = \frac{13-\sqrt{69}}{10} \approx 0.47$, it was recently established that there exist multiplicative Sidon sets (sets without any non-trivial solutions to $ab = cd$) in $\{1, 2, \ldots, n\}$ with maximal gap size $\ll_{\varepsilon} n^{\rho + \varepsilon}$. Here we improve upon this result and show that one can take $\rho = \frac{10}{33} \approx 0.303$ instead.
\end{abstract}

\maketitle

\section{Introduction}
Define a set $A$ of positive integers to be a multiplicative Sidon set if the equality of products $ab = cd$ with $a,b,c,d \in A$ implies the equality of sets $\{a, b\} = \{c, d\}$. In other words, all products of two elements from $A$ are essentially unique. Now, one can verify that the set of primes is a multiplicative Sidon set, and it is natural to ask if such a set can get much larger than that. Erd\H{o}s~\cite{Erdos1938, Erdos1968} almost completely answered this question, by showing that there exist absolute constants $c_1$ and $c_2$ such that for all $n \ge 2$ the largest possible multiplicative Sidon subset of $[n] := \{1, 2, \ldots, n\}$ has size between $\pi(n) + \frac{c_1 n^{3/4}}{(\log n)^{3/2}}$ and $\pi(n) + \frac{c_2 n^{3/4}}{(\log n)^{3/2}}$, where $\pi(n)$ is the number of primes in $[n]$. 

Since it is conjectured that, for every $\varepsilon > 0$, the difference between consecutive primes $p$ and $q$ is smaller than $p^{\varepsilon}$ if $p$ is large enough, the set of primes in $[n]$ conjecturally gives a multiplicative Sidon set in $[n]$ whose maximal gap between consecutive elements is $\ll_{\varepsilon} n^{\varepsilon}$. Unfortunately, proving such bounds on prime gaps is infamously intractable with current knowledge, and even under the assumption of the Riemann Hypothesis the best one can show is a gap size of $\ll_{\varepsilon} n^{1/2 + \varepsilon}$. 

This lack of understanding of prime gaps prompted S\'ark\"ozy~\cite{Sarkozy2001} to ask whether one can show the existence of a multiplicative Sidon subset of $[n]$ with gaps at most $\sqrt{n}$, which was recently answered in the affirmative by the authors and Monticone~\cite{DMT2026}. More precisely, let us define $g(n)$ as the infimum of all real numbers $L$ such that there exists a multiplicative Sidon subset of $[n]$ which has non-empty intersection with every sub-interval $[x, x+L] \subseteq [1, n]$ of length $L$. In \cite{DMT2026} it was then shown that for every $\varepsilon > 0$ we have $g(n) \ll_{\varepsilon} n^{ \frac{13-\sqrt{69}}{10} + \varepsilon}$. 

In this paper we further improve the best known upper bound on $g(n)$, by combining the probabilistic method with recent results of Gafni and Tao~\cite{GafniTao2025} on the sparsity of intervals with exceptionally few primes. This in particular lowers the upper bound to $g(n) \ll_{\varepsilon} n^{10/33 + \varepsilon}$, and moreover shows how any further improvements to results similar to those in~\cite{GafniTao2025} immediately feed into improved bounds on $g(n)$.

\subsection*{Declaration of AI usage}
The proof of the upper bound $g(n) \ll_{\varepsilon} n^{10/33 + \varepsilon}$ was found by ChatGPT 5.5 Pro, and it deserves credit for most of the main ideas that we present. The proof was subsequently simplified and generalized by the authors, and the end result is completely human-written. For full transparency, the entire paper that ChatGPT initially produced is publicly available; see~\cite{GPT}. The same GitHub repository also contains fully self-contained Lean formalizations of Corollary \ref{firstbound} and Theorem \ref{main} that we prove in Sections \ref{sec:lll} and \ref{sec:many_primes_imply_small_gaps} respectively. The formalizations of these results were obtained by the automated theorem proving tool Aristotle from Harmonic~\cite{ari}.

\section{An application of the Lov\'asz local lemma}\label{sec:lll}

It is interesting to note that~\cite{DMT2026, Erdos1938, Erdos1968} all make use of graph-theoretic language and results in order to prove bounds on multiplicative Sidon sets, which are themselves defined as a purely number-theoretical object. Here we follow this tradition. 

Recall that, if $E$ is a vertex in a graph $G$, then $\Gamma_G(E)$ denotes the neighborhood of $E$, i.e. the set of all vertices in $G$ that $E$ is connected to. With this in mind, we then define the following probabilistic notion.

\begin{definition}
If $\mathcal{E}$ is a set of events in a probability space, then a graph $G$ is called a \emph{dependency graph} on $\mathcal{E}$, if $G$ has the events as vertices, and every event $E \in \mathcal{E}$ is mutually independent of $\mathcal{E} \setminus \big(\{E\} \cup \Gamma_G(E)\big)$. 
\end{definition}

Using the above definition we can state the celebrated Lov\'asz local lemma~\cite{ErdosLovasz1975, Spencer1977} as follows, where $\Pr(E)$ denotes the probability of an event $E$.

\begin{lemma}[asymmetric Lov\'asz local lemma] \label{lll}
Let $\mathcal{E}$ be a finite set of `bad' events in a probability space, and let $G$ be a corresponding dependency graph on $\mathcal{E}$. If there exists a function $f : \mathcal{E} \to [0, 1)$ such that for all $E \in \mathcal{E}$ we have 

\begin{equation} \label{LLLeq}
\Pr(E) \le f(E) \prod_{E' \in \Gamma_G(E)} \big(1 - f(E')\big),
\end{equation}

then the probability that no bad events happen is positive.
\end{lemma}

As it turns out, Lemma \ref{lll} can be applied to show the existence of certain multiplicative Sidon sets.

\begin{lemma} \label{randomMS}
For real numbers $\alpha$ and $\beta$ that satisfy $0 < \beta < 2\alpha$, let $\mathcal{S} = \{S_1, S_2, \ldots\}$ be a collection of disjoint subsets of $[n]$, with $|\mathcal{S}| \le n^{\beta}$ and $|S_i| \ge n^{\alpha}$ for all $i$. Then for all sufficiently large $n$ there exists a multiplicative Sidon set $C = \{c_1, c_2, \ldots\}$ with $c_i \in S_i$ for all $i$.
\end{lemma}

\begin{proof}
Writing $\tau(n)$ for the number of divisors of $n$, let $\varepsilon := \frac{2\alpha - \beta}{2} > 0$ and assume that $n$ is large enough so that both $n^{\varepsilon} \ge 48$ and $\tau(m) < n^{\varepsilon}$ hold for all $m \le n^2$.\footnote{Using the explicit results from~\cite{NicolasRobin1983}, one can check that any $n \ge e^{e^{3\varepsilon^{-1}}}$ is admissible.} For every $i$, let $X_i$ denote the random variable where we choose an element from $S_i$ independently and uniformly at random. The goal is to prove that there is a positive probability that this gives a multiplicative Sidon set. 

For every $a \in S_i, b \in S_j, c \in S_k, d \in S_l$ with $ab = cd$ and $\{a, b\} \neq \{c, d\}$, let $E^{i,j,k,l}_{a,b,c,d}$ denote the bad event that $X_i = a, X_j = b, X_k = c, X_l = d$. For any valid permutation of the indices choose one representative, and write $\mathcal{E}$ as the collection of the representatives of all bad events. We then construct the graph $G$ with vertices the elements from $\mathcal{E}$, and where we draw an edge between $E^{i,j,k,l}_{a,b,c,d}$ and $E^{i',j',k',l'}_{a',b',c',d'}$ if $\{i,j,k,l\} \cap \{i',j',k',l'\} \neq \emptyset$. This gives a dependency graph, as elements from distinct sets in $\mathcal{S}$ are chosen independently. Now, with $f(E) := 2 \Pr(E)$ defined for any bad event $E \in \mathcal{E}$, by Lemma \ref{lll} it then suffices to prove that for all $E$ the inequalities $f(E) < 1$ and \eqref{LLLeq} hold. 

For the first inequality, we note that any non-trivial relation $ab = cd$ contains at least three distinct variables, which implies $$f(E) = 2 \Pr(E) \le \frac{2}{n^{3\alpha}} < \frac{2}{n^{\varepsilon}} < 1.$$ As for inequality \eqref{LLLeq}, let us fix $i$ and define $P_i$ to be the sum of the probabilities of all bad events involved with $X_i$. We then note that the bad events involving an $a \in S_i$ come in three distinct types: either $a^2 = bc$, or $ab = c^2$, or $ab = cd$, where $a,b,c$ and possibly $d$ are all from different subsets. The contribution to $P_i$ of the bad events of the latter type is at most
\begin{align*}
&\sum_{a\in S_i} \Pr(X_i = a) \sum_j \sum_{b\in S_j} \Pr(X_j = b) \sum_{k,l} \sum_{\substack{c \in S_k, d \in S_l \\ ab=cd}} \Pr(X_k = c)\Pr(X_l = d) \\
&\le \sum_{a\in S_i} \Pr(X_i = a) \sum_j \sum_{b\in S_j} \Pr(X_j = b) \frac{\tau(ab)}{n^{2\alpha}} \\
&< \frac{n^{\varepsilon}}{n^{2\alpha}} \sum_{a\in S_i} \Pr(X_i = a) \sum_j \sum_{b\in S_j} \Pr(X_j = b) \\
&= \frac{|\mathcal{S}|}{n^{2\alpha - \varepsilon}}.
\end{align*}

Similarly, the contribution of bad events of the first type is bounded by
\begin{align*}
&\sum_{a\in S_i} \Pr(X_i = a) \sum_{j,k} \sum_{\substack{b \in S_j, c \in S_k \\ a^2=bc}} \Pr(X_j = b)\Pr(X_k = c) \\
&\le \sum_{a\in S_i} \Pr(X_i = a) \frac{\tau(a^2)}{n^{2\alpha}} \\
&< \frac{1}{n^{2\alpha - \varepsilon}},
\end{align*}

while events of the second type contribute at most
\begin{align*}
&\sum_k \sum_{c\in S_k} \Pr(X_k = c) \sum_j \sum_{\substack{a \in S_i, b \in S_j \\ ab=c^2}} \Pr(X_i = a)\Pr(X_j = b) \\
&\le \sum_k \sum_{c\in S_k} \Pr(X_k = c) \frac{\tau(c^2)}{n^{2\alpha}} \\
&< \frac{|\mathcal{S}|}{n^{2\alpha - \varepsilon}}.
\end{align*}

Combining these estimates we see that for every fixed $i$, $$P_i < \frac{2|\mathcal{S}| + 1}{n^{2\alpha - \varepsilon}} \le \frac{3n^{\beta}}{n^{2\alpha - \varepsilon}} = \frac{3}{n^{\varepsilon}} \le \frac{1}{16}.$$ Now using the above, for every bad event $E = E^{i,j,k,l}_{a,b,c,d}$ we get
\begin{align*}
&f(E) \prod_{E' \in \Gamma_G(E)} \big(1 - f(E')\big) \\
&\ge f(E)\big(1 - \sum_{E' \in \Gamma_G(E)} f(E')\big) \\
&\ge f(E)\big(1 - 2(P_i + P_j + P_k + P_l)\big) \\
&\ge \frac{1}{2}f(E) \\
&= \Pr(E),
\end{align*}

finishing the proof.
\end{proof}

For an immediate application of Lemma \ref{randomMS}, let $\varepsilon \in \big(0, \frac{2}{3}\big)$ be arbitrary, choose $\alpha := \frac{1}{3} + \varepsilon$, $\beta := \frac{2}{3} - \varepsilon$ and, with $H := \left \lceil n^{\alpha} \right \rceil$, define $S_i := \big((i-1) H, i H]$ for $i = 1, 2, \ldots, \left \lfloor \frac{n}{H} \right \rfloor$. One can then verify that the conditions of Lemma \ref{randomMS} are met, so that we may deduce the existence of a multiplicative Sidon set that contains an element from each $S_i$ if $n$ is large enough. In particular, such a multiplicative Sidon set intersects every interval of length $2H$, giving the following corollary.

\begin{corollary} \label{firstbound}
For all $\varepsilon > 0$ we have $g(n) \ll_{\varepsilon} n^{1/3 + \varepsilon}$.
\end{corollary}

\section{Bounding gaps via the scarcity of prime-poor intervals} \label{sec:many_primes_imply_small_gaps}
We introduce the following notion that quantifies how often short intervals contain very few primes.

\begin{definition} \label{lebesgue}
For a positive real number $\alpha$, define $\lambda(\alpha)$ to be the smallest non-negative real such that, for all $\varepsilon > 0$ and all large enough $n$, the Lebesgue measure of the set of $x \in [0, n]$ for which the interval $(x, x+x^{\alpha}]$ contains fewer than $5^{1/\alpha}$ primes, is smaller than $n^{\lambda(\alpha) + \varepsilon}$.
\end{definition}

We certainly expect $\lambda(\alpha)$ to be $0$ for all $\alpha > 0$. Indeed, it is even conjectured that the prime number theorem holds in such intervals, whereas we only require a (sufficiently large) finite number of primes. However, it is currently only known that $\lambda(\alpha) = 0$ holds for $\alpha > 0.52$, thanks to Li~\cite{Li2025}. Fortunately, we do not quite need $\lambda(\alpha) = 0$ for our purposes.

\begin{theorem} \label{main}
If $\lambda(\alpha) < 3\alpha$ for some $\alpha > 0$, then $g(n) \ll_\alpha n^{\alpha}$.
\end{theorem}

\begin{proof}
Set $\varepsilon := \alpha - \frac{\lambda(\alpha)}{3} > 0$, define $\beta := 2\alpha - \varepsilon$, and assume that $n$ is sufficiently large. With $H := 2\lceil n^{\alpha} \rceil$, we consider the intervals $\big((i-1)H, iH]$ for positive integers $i \le \frac{n}{H}$. Apart from the first interval corresponding to $i = 1$, every other interval then either contains at least $L := \left \lceil 5^{1/\alpha} \right \rceil$ primes $p$ with $H < p \le n$, or it contributes measure at least $n^{\alpha}$ to the total measure of intervals counted in Definition \ref{lebesgue}, which is bounded by $n^{\lambda(\alpha) + \varepsilon}$. Hence, the number of exceptional intervals which have fewer than $L$ primes $p$ with $H < p \le n$ is at most $n^{\lambda(\alpha) + \varepsilon - \alpha} + 1 \le n^{\beta}$. Let us call these exceptional intervals \emph{bad} and the other intervals \emph{good}, and let $N$ be the number of good intervals. 

The goal is to show that it is possible to choose a single prime from the first $L$ primes in every good interval, in such a way that every bad interval contains more than $\frac{H}{2}$ integers which are not divisible by any of the chosen primes. If we choose the primes in such a way that a bad interval $I$ contains at least $\frac{H}{2}$ integers divisible by one or more of the chosen primes, we say that this choice \emph{ruins} $I$. Since every prime we consider is larger than $H$, it can divide at most one integer in every interval $I$. We may therefore assume $N \ge \frac{H}{2}$, as otherwise it would not even be possible to ruin any bad interval. 

Now, every integer in a bad interval is divisible by at most $\frac{1}{\alpha}$ distinct primes $p > H > n^{\alpha}$, as otherwise the product of these primes would be too large. So for a bad interval $I$, there are at most $\frac{H}{\alpha}$ candidate primes that divide an integer in $I$. In particular, there are certainly no more than $2^{\frac{H}{\alpha}}$ ways to choose $\frac{H}{2}$ primes that would ruin any fixed $I$. The primes from the remaining $N - \frac{H}{2}$ good intervals can then still be chosen in $L$ ways each. Since the number of bad intervals is bounded by $n$, the total number of ways to ruin any bad interval is at most $$n 2^{\frac{H}{\alpha}} L^{N - \frac{H}{2}} \le n \left(\frac{2}{\sqrt{5}} \right)^{H/\alpha} L^N < L^N.$$ As the total number of ways to choose primes from good intervals is exactly $L^N$, there exists a choice that does not ruin any bad interval. Fix such a choice and let $B$ be the set of primes chosen. 

For a bad interval $I$, let $S \subseteq I$ be the subset of integers not divisible by any prime from $B$, and let $\mathcal{S}$ be the collection of all sets $S$. By the above, we then have $|S| > \frac{H}{2} \ge n^{\alpha}$, while $|\mathcal{S}| \le n^{\beta}$ with $\beta = 2\alpha - \varepsilon < 2\alpha$. We may therefore apply Lemma \ref{randomMS}, and deduce the existence of a multiplicative Sidon set $C$ that consists of one element from every bad interval. We now claim that $A := B \cup C$ is a multiplicative Sidon set that intersects every interval of length $2H \ll n^{\alpha}$. The latter part is straightforward, as $A$ contains exactly one element from every interval, so that the gap between consecutive elements (as well as the initial and final gap) is at most $2H$. The multiplicative Sidon property is equally easy to verify, as any non-trivial equality of the form $ab = cd$ with $a,b,c,d \in A$ must contain elements from $B$ by the multiplicative Sidon property of $C$. But as the primes are all distinct and do not divide any elements in $C$ by construction, such a prime must then occur on both sides of the equality, from which we deduce $\{a, b\} = \{c, d\}$.
\end{proof}

\section{Applying prime results from the literature}
By Corollary~\ref{firstbound}, a random construction gives $g(n) \ll_{\varepsilon} n^{1/3 + \varepsilon}$. This corollary can be viewed as a special case of Theorem \ref{main}, where the set of primes $B$ is empty and one only uses the trivial bound $\lambda(\alpha) \le 1$. There are, however, various results from the literature that imply that $\lambda(\alpha) < 3\alpha$ holds, even for some $\alpha \le \frac{1}{3}$. 

For an example that does not quite work but still improves on the results in~\cite{DMT2026}, if $p_1, p_2, \ldots$ denotes the sequence of primes, then J\"arviniemi~\cite{Jarviniemi2022} proved the upper bound $$\sum_{\substack{p_m \in [n, 2n] \\ p_{m+1} - p_m\ge n^{0.45}}} (p_{m+1} - p_m) \ll_{\varepsilon} n^{0.63 + \varepsilon},$$ which gives $\lambda(0.45) \le 0.63 < 3 \cdot 0.45$, implying $g(n) \ll n^{0.45}$. 

Moving from a sum of differences to a sum of squared differences, Stadlmann~\cite{Stadlmann2022} showed that $$\sum_{p_m \le n} (p_{m+1} - p_m)^2 \ll_{\varepsilon} n^{1.23 + \varepsilon}.$$ One can check that from this estimate it follows that $\lambda(\alpha) \le 1.23 - \alpha$, so that $\lambda(\alpha) < 3\alpha$ for any $\alpha > \frac{1.23}{4} = 0.3075$. That is, $g(n) \ll_{\varepsilon} n^{0.3075 + \varepsilon}$, genuinely improving on Corollary~\ref{firstbound}. Furthermore, as the Lindel\"of Hypothesis implies the same bound on the sum of the squared differences but with $1.23$ replaced by $1$ (see~\cite{Yu1996}), conditional on the Lindel\"of Hypothesis this gives us $g(n) \ll_{\varepsilon} n^{1/4 + \varepsilon}$. 

Perhaps not surprisingly, this very same conditional result also follows via a slightly different route. Using notation which is essentially equivalent to \cite[Definition 1.1]{GafniTao2025}, we introduce the following strengthened version of $\lambda(\alpha)$.

\begin{definition} \label{lebesguetwo}
For a positive real number $\alpha$, define $\mu(\alpha)$ to be the smallest non-negative real such that, for all $\varepsilon > 0$ and all large enough $n$, the Lebesgue measure of the set of $x \in [0, n]$ for which the prime number theorem fails to hold in the interval $(x, x+x^{\alpha}]$, is smaller than $n^{\mu(\alpha) + \varepsilon}$.
\end{definition}

It then follows from the Lindel\"of Hypothesis (see \cite[Lemma 4]{Bazzanella2009}, or \cite[Theorem 2 (ii)]{BazzanellaPerelli2000} if we instead assume the Riemann Hypothesis) that $\mu(\alpha) \le 1 - \alpha$ for all $\alpha \in [0, 1]$. In particular, taking $\alpha = \frac{1}{4} + \varepsilon$ for arbitrarily small $\varepsilon$ once again gives $$\lambda(\alpha) \le \mu(\alpha) \le \frac{3}{4} - \varepsilon < 3 \alpha,$$ which also implies the conditional $g(n) \ll_{\varepsilon} n^{1/4 + \varepsilon}$ by Theorem~\ref{main}. 

As for the best unconditional bound, we can use results of Gafni and Tao \cite{GafniTao2025} to slightly improve the $0.3075$ exponent that we deduced from \cite{Stadlmann2022}.

\begin{theorem}
For all $\varepsilon > 0$ we have $g(n) \ll_{\varepsilon} n^{10/33 + \varepsilon}$.
\end{theorem}

\begin{proof}
By Theorem \ref{main} and the trivial inequality $\lambda(\alpha) \le \mu(\alpha)$, it suffices to show that for all $\alpha > \frac{10}{33}$ we have $\mu(\alpha) \le \frac{10}{11}$. Unpacking and applying the definitions used in \cite[Theorem 1.2]{GafniTao2025} shows that we need to look at the supremum of $$\frac{23}{33}(1 - \sigma)A(\sigma) + 2\sigma - 1$$ over all $\sigma \in [0, 1)$ for which $A(\sigma) \ge \frac{33}{23}$. By using \cite[Table 1]{GafniTao2025}, one can verify that $A(\sigma) < \frac{33}{23}$ holds whenever $\sigma > \frac{10}{11}$, so that we may restrict to $\sigma \le \frac{10}{11}$, i.e. the first $14$ rows of \cite[Table 1]{GafniTao2025}. Manual verification of these $14$ ranges does indeed give $$\frac{23}{33}(1 - \sigma)A(\sigma) + 2\sigma - 1 \le \frac{10}{11}$$ for all $\sigma \le \frac{10}{11}$ (where the maximum of the function on the left-hand side is reached at the boundary point $\sigma = \frac{10}{11}$), as required.
\end{proof}

\end{document}